\documentclass[10pt,reqno,a4paper]{amsart}

\let\oldtocsection=\tocsection

\let\oldtocsubsection=\tocsubsection

\let\oldtocsubsubsection=\tocsubsubsection

\renewcommand{\tocsection}[2]{\hspace{0em}\oldtocsection{#1}{#2}}
\renewcommand{\tocsubsection}[2]{\hspace{2em}\oldtocsubsection{#1}{#2}}
\renewcommand{\tocsubsubsection}[2]{\hspace{2em}\oldtocsubsubsection{#1}{#2}}

\usepackage{amsmath,amsthm,amsfonts,amssymb}

\usepackage{marginnote}

\usepackage{fancyhdr}

\usepackage{graphicx}

\usepackage{pdfpages}

\usepackage{tikz}
\usetikzlibrary{matrix}
\usetikzlibrary{cd}

\usepackage{xcolor}

\numberwithin{figure}{section}
\numberwithin{equation}{section}

\newtheorem{thm}{Theorem}[section]

\newtheorem{defn}[thm]{Definition}

\newtheorem{lmm}[thm]{Lemma}

\newtheorem{prp}[thm]{Proposition}

\newtheorem{remark}[thm]{Remark}

\newcommand{\tei}{Teichm\"uller}
\newcommand{\qc}{quasiconformal}

\newcommand{\idt}{of $\id$-type}
\newcommand{\nd}{\mathcal{N}_d}

\newcommand{\dist}{\operatorname{dist}}
\newcommand{\id}{\operatorname{id}}
\newcommand{\Id}{\operatorname{Id}}
\newcommand{\SV}{\operatorname{SV}}
\newcommand{\Crit}{\operatorname{Crit}}

\renewcommand{\Re}{\operatorname{Re\,}}
\renewcommand{\Im}{\operatorname{Im\,}}

\newcommand{\abs}[1]{\left| #1 \right|}

\newcounter{reminder}

\graphicspath{ {./images/}}

\title{Infinite-dimensional Thurston theory\\and transcendental dynamics III:\\entire functions with escaping singular orbits\\in the degenerate case}
\author{Konstantin Bogdanov}

\begin{document}

\begin{abstract}
	We classify transcendental entire functions that are compositions of a polynomial and the exponential for which all singular values escape on disjoint rays. We focus on the case where the escape is degenerate in the sense that points from different singular orbits are arbitrarily close to each other. As in the general case we employ an iteration procedure on an \emph{infinite-dimensional} Teichm\"uller space, analogously to the Thurston's classical Topological Characterization of Rational Functions, but for an \emph{infinite} set of marked points.
\end{abstract}

\maketitle
	
\tableofcontents

\addtocontents{toc}{\protect\setcounter{tocdepth}{1}}

This is the third out of four articles publishing the results of the author's doctoral thesis \cite{MyThesis}. The machinery needed for the Thurston iteration on an infinite-dimensional \tei\ space is developed in \cite{IDTT1}. In this article and in \cite{IDTT2} we classify transcendental entire functions that are compositions of a polynomial and the exponential for which their singular values escape on disjoint dynamic rays. \cite{IDTT2} covers the general case of such escape when the distance between points on different singular orbits is bounded below, whereas in this paper we get rid of this restriction. In the fourth article we investigate continuity of such families of functions with respect to potentials and external addresses and use the continuity argument to extend our classification to the case of escape on the (pre-)periodic rays.

\section{Introduction}

In transcendental dynamics, as for polynomials, an important dynamically natural subset of the complex plane is the \emph{escaping set}. For a transcendental entire function $f$, its \emph{escaping set} is defined as
$$I(f)=\{z\in\mathbb{C}: f^n(z)\to\infty\text{ as }n\to\infty\}.$$
Since the Julia set is the boundary of $I(f)$, the Julia and Fatou sets can be reconstructed from $I(f)$, which is thus one of the most fundamental invariant sets for the dynamics of entire functions. It is proved in \cite{RRRS} for transcendental entire functions of bounded type and finite order that the escaping set of every function in the family is organized as the union \emph{dynamic rays} and their endpoints. About the escaping points that belong to a dynamic ray one says that they \emph{escape on rays}. This is the standard mode of escape for many important families of transcendental entire functions.

In \cite{IDTT2} we classified entire functions of the form $p\circ\exp$ where $p$ is a monic polynomial such that all their singular points escape on disjoint rays (i.e.\ every dynamic ray contains no more than one post-singular point) and distance between points on different singular orbits is bounded below. Denote
$$\mathcal{N}_d:=\{p\circ\exp: p\text{ is a monic polynomial of degree }d\}.$$ 
As for the exponential family \cite{SZ-Escaping}, for functions of the form $p\circ\exp$ the points escaping on rays can be described by their potential (or ``speed of escape'') and external address (or ``combinatorics of escape'', i.e.\ the sequence of dynamic rays containing the escaping orbit). So the classification was accomplished in terms of potentials and external addresses. To be more precise, for every function $f=p\circ\exp$ the classification was obtained within the \emph{parameter space} of $f$. The parameter space of an entire function $g$ is the set of entire functions $\varphi\circ f\circ\psi$ where $\varphi,\psi$ are \qc\ self-homeomorphisms of $\mathbb{C}$. For the exponential family $\mathcal{N}_1$ this definition of the parameter space coincides with the parametrization of $e^z+\kappa$ (up to affine conjugation). Earlier classification results for $\mathcal{N}_1$ were obtained in \cite{FRS,MarkusParaRays,MarkusThesis}.

In this article we want to extend this classification to the case when points from different singular orbits come arbitrarily close to each other, that is, to prove the following more general version of \cite[Classification~Theorem~1.4]{IDTT2}.

\begin{thm}[Classification Theorem]
	\label{thm:main_thm}
	Let $f_0\in\mathcal{N}_d$ be a transcendental entire function with singular values $\{v_i\}_{i=1}^m$. Let also $\{\underline{s}_i\}_{i=1}^m$ be exponentially bounded external addresses that are non-(pre-)periodic and non-overlapping, and $\{T_i\}_{i=1}^m$ be real numbers such that $T_i>t_{\underline{s}_i}$. Then in the parameter space of $f_0$ there exists a unique entire function $g=\varphi\circ f_0\circ\psi$ such that each of its singular values $\varphi(v_i)$ escapes on rays, has potential $T_i$ and external address $\underline{s}_i$.
	
	Conversely, every function in the parameter space of $f_0$ such that its singular values escape on disjoint rays is one of these. 
\end{thm}

Precise definitions of terms used the Classification Theorem are given in Subsection~\ref{subsec:escaping_set}. In particular, the condition that external addresses are non-(pre-)periodic and non-overlapping means that the singular values of $g$ escape on disjoint rays.

The second paragraph of the Classification Theorem~\ref{thm:main_thm} follows immediately from Theorem~\ref{thm:as_formula}, and we added it to the statement in order to show that we indeed have a \emph{classification} of function in $\mathcal{N}_d$ with singular values escaping on rays (subject to some restrictions).

We use the same general scheme of the proof developed in \cite{MyThesis,IDTT1,IDTT2} involving a special generalization of the \emph{Thurston's Topological Characterization of Rational Functions} \cite{DH,HubbardBook2} for infinitely many marked points.
\begin{enumerate}
	\item Construct a (non-holomorphic) ``model map'' $f$ for which the singular values escape on rays with the desired potential and external address. It will define the Thurston $\sigma$-map acting on the \tei\ space $\mathcal{T}_f$ of $\mathbb{C}$ with $P_f$ removed, where $P_f$ is the postsingular set of $f$ (Definition~\ref{defn:tei_space}).
	\item Construct a compact subset $\mathcal{C}_f$ of $\mathcal{T}_f$ which is invariant under $\sigma$.
	\item Prove that $\sigma$ is strictly contracting  in the \tei\ metric on $\mathcal{C}_f$.
	\item Prove that $\sigma$ has a unique fixed point in $\mathcal{C}_f$, and this point corresponds to the entire function with the desired conditions on its singular orbits. 
\end{enumerate}

Items (1),(3) and (4) and a big part of (2) are almost completely proved in \cite{IDTT1,IDTT2}, so in this article we will mainly be focused on the upgrade of (2).

The compact invariant subset $\mathcal{C}_f$ for the non-degenerate case is constructed in \cite{IDTT2} approximately as follows. For some initially chosen constant $\rho>0$, it is the set of points represented by \qc\ maps $\varphi$ that can be obtained from the identity via an isotopy $\varphi_u$ such that: 
\begin{enumerate}
	\item for very point $z\in P_f$ outside of $\mathbb{D}_\rho(0)$, $\varphi_u(z)$ is contained in a small disk $U_z$ around $z$ so that the disks $U_z$ are mutually disjoint; 
	\item inside of $\mathbb{D}_\rho(0)$ the distances between marked points $\varphi_u(z)$ are bounded from below;
	\item the isotopy type of $\varphi_u$ relative points of $P_f$ outside of $\mathbb{D}_\rho(0)$ is ``almost'' the isotopy type of identity; 
	\item the isotopy type of $\varphi_u$ relative points of $P_f$ inside of $\mathbb{D}_\rho(0)$ is not ``too complicated'', i.e.\ there are quantitative bounds on how many times the marked points ``twist'' around each other.
\end{enumerate}

Conditions (2),(3) create a separation of the plane into the two subsets: $\mathbb{D}_\rho(0)$ without much control on the behavior of $\varphi_u$ but containing only finitely many marked points, and the complement of $\mathbb{D}_\rho(0)$ where the homotopy information is trivial but there are infinitely many marked points.

However, it is not enough to consider such $\mathcal{C}_f$ for the case when points belonging to different singular orbits come arbitrarily close to each other: it might be simply impossible to construct such disjoint disks $U_z$ so that they stay invariant under the $\sigma$-map (see \cite[Remark~5.10]{IDTT2}). Instead, outside of $\mathbb{D}_\rho(0)$ we consider \emph{clusters} of points (Definition~\ref{defn:cluster}) --- finite subsets of $P_f$ formed by points that are ``very close'' to each other. Therefore, in condition (1) we require that the disjoint disks $U_z$ contain clusters of points rather than a single point and add the condition
\begin{enumerate}
	\item [(5)] If $U_z$ contains more than one point, then these points ``almost'' do not move relative to each other. 
\end{enumerate}
What exactly this ``almost'' means is described in Proposition~\ref{prp:good_big_disk_around_origin_cluster_estimates}.

Another rather technical addition comparing to the older $\mathcal{C}_f$ from \cite{IDTT2} is that we also need to partially control the clusters inside of $\mathbb{D}_\rho(0)$, that is, we need an additional estimate like in (2) to prevent the marked points come too close to the asymptotic value --- otherwise after one iteration of the $\sigma$ map they might be ``thrown'' towards $-\infty$, i.e.\ outside of $\mathbb{D}_\rho(0)$. The construction of $\mathcal{C}_f$ in full generality is the content of Theorem~\ref{thm:invariant_subset}.

\section{Prerequisites}

In this section we assemble some statements and definitions mainly from \cite{IDTT1,IDTT2} that will be needed in our constructions.

\subsection{Escaping set of functions in $\mathcal{N}_d$}

\label{subsec:escaping_set}

It follows from \cite{RRRS} that for every function $f\in\nd$ its escaping set is modeled on dynamic rays. In particular, that every escaping point is mapped on a dynamic ray after finitely many iterations. A related notion is a ray tail.

\begin{defn}[Ray tails]
	\label{dfn:ray_tail}
	Let $f$ be a transcendental entire function. A \emph{ray tail} of $f$ is a continuous curve $\gamma:[0,\infty)\to I(f)$ such that for every $n\geq 0$ the restriction $f^n|_\gamma$ is injective with $\lim_{t\to\infty}f^n(\gamma(t))=\infty$, and furthermore $f^n(\gamma(t))\to\infty$ uniformly in $t$ as $n\to\infty$.
\end{defn}

\begin{defn}[Dynamic rays, escape on rays]
	\label{dfn:dynamic_ray}
	A \emph{dynamic ray} of a transcendental entire function $f$ is a maximal injective curve $\gamma:(0,\infty)\to I(f)$ such that $\gamma|_{[t,\infty)}$ is a ray tail for every $t>0$.
	
	If $z\in I(f)$ belongs to a dynamic ray, we say that $z$ \emph{escapes on rays}. 
\end{defn}

We showed in \cite{IDTT2} that for functions in $\nd$ its dynamic rays (at least their parts near $\infty$) are characterized by an exponentially bounded external address and parametrized by potential.

Fix an integer $d>0$, and denote by $F:\mathbb{R^+}\to\mathbb{R^+}$ the function $$F(t):=\exp(dt)-1.$$
The \emph{external address} $\underline{s}=(s_0 s_1 s_2 ... )$ is an infinite sequence of integers $s_0,s_1,s_2,...$ On the set of all external addresses we can consider the standard shift-operator $\sigma:(s_0 s_1 s_2 ... )\mapsto (s_1 s_2 s_3 ... )$. Two external addresses $\underline{s}_1$ and $\underline{s}_2$ are called \emph{overlapping} if there are integers $k,l\geq0$ so that $\sigma^k\underline{s}_1=\sigma^l\underline{s}_2$.

Further, an external address $\underline{s}=(s_0 s_1 s_2 ... )$ is called \emph{exponentially bounded} if there exists $t>0$ such that $s_n/{F^n(t)}\to 0$ as $n\to\infty$. The infimum of such $t$ is denoted by $t_{\underline{s}}$.	

As can be seen in the following theorem, exponentially bounded external addresses correspond to dynamic rays (at least to their part near $\infty$) of functions in $\nd$. Moreover, from the formula~\ref{eqn:as_formula}, dynamic rays look near $\infty$ almost like straight horizontal lines with the imaginary part that can be recovered from their external address.

\begin{thm}[Escape on rays in $\nd$ \cite{IDTT2}]
	\label{thm:as_formula}
	Let $f\in\mathcal{N}_d$. Then for every exponentially bounded external address there exists a dynamic ray realizing it.
	
	If $\mathcal{R}_{\underline{s}}$ is dynamic ray having an exponentially bounded external address $\underline{s}=(s_0 s_1 s_2 ... )$, and no strict forward iterate of $\mathcal{R}_{\underline{s}}$ contains a singular value of $f$, then $\mathcal{R}_{\underline{s}}$ is the unique dynamic ray having external address $\underline{s}$, and it can be parametrized by $t\in(t_{\underline{s}},\infty)$ so that 
	\begin{equation}
		\label{eqn:as_formula}
		\mathcal{R}_{\underline{s}}(t)=t+\frac{2\pi i s_0}{d} + O(e^{-t/2}),
	\end{equation}
	and $f^n\circ \mathcal{R}_{\underline{s}}=\mathcal{R}_{\sigma^n \underline{s}}\circ F^n$. Asymptotic bounds $O(.)$ for $\mathcal{R}_{\sigma^n\underline{s}}(t)$ are uniform in $n$ on every ray tail contained in $\mathcal{R}_{\underline{s}}$.
	
	If none of the singular values of $f$ escapes, then $I(f)$ is the disjoint union of dynamic rays and their escaping endpoints.
\end{thm}

The parametrization~\ref{eqn:as_formula} allows us to define the potential of a point on a dynamic ray: if $z\in\mathcal{R}_{\underline{s}}$ corresponds to the parameter $t$ of formula~\ref{eqn:as_formula}, we say that $t$ is the \emph{potential} of $z$. From Theorem~\ref{thm:as_formula} follows immediately that the potential is defined in a canonical way, and if the potential of $z$ is equal to $t$, then the potential of $f(z)$ is equal to $F(t)$. Note that different points on the same ray have different potentials.

\subsection{Setup of Thurston iteration and contraction}

\label{subsec:setup_and_contraction}

In this subsection we overview a few important definitions from \cite{IDTT2} and sketch the setup of the Thurston iteration.

One starts with the construction of a quasiregular function $f$ modeling escaping behavior of singular values. Let $f_0\in\mathcal{N}_d$, and $v_1,...,v_m$ be the finite singular values of $f_0$. Further, let $\mathcal{O}_1=\{a_{1j}\}_{j=0}^\infty,...,\mathcal{O}_m=\{a_{mj}\}_{j=0}^\infty$ be some orbits of $f_0$ escaping on disjoint rays $\mathcal{R}_{ij}$, and let $R_{ij}$ be the part of the ray $\mathcal{R}_{ij}$ from $a_{ij}$ to $\infty$ (containing also $a_{ij}$).

Now, we introduce the \emph{capture}. Choose some bounded Jordan domain $U\subset\mathbb{C}\setminus\bigcup_{\substack{i=\overline{1,m}\\j=\overline{1,\infty}}} R_{ij}$ containing all singular values of $f_0$ and the first point $a_{i0}$ on each escaping orbit $\mathcal{O}_i$. Define an isotopy $c_u:\mathbb{C}\to\mathbb{C}$ through \qc\ maps, such that $c_0=\id$, $c_u=\id$ on $\mathbb{C}\setminus U$, and for each $i=\overline{1,m}$ we have $c_1(v_i)=a_{i0}$. Denote $c=c_1$. Thus, $c$ is a \qc\ homeomorphism mapping singular values to the first points on the orbits $\mathcal{O}_i$.

Define a function $f:=c\circ f_0$. It is a quasiregular function whose singular orbits coincide with $\{\mathcal{O}_i\}_{i=1}^m$. Its post-singular set of $f$ is defined in a standard way: $$P_f:=\cup_{i=1}^m\mathcal{O}_i.$$ We also call $P_f$ the set of \emph{marked points}.

The Thurston $\sigma$-map acts on the \tei\ space of $\mathbb{C}\setminus P_f$.

\begin{defn}[\tei\ space of $\mathbb{C}\setminus P_f$]
	\label{defn:tei_space}
	The \emph{\tei\ space} $\mathcal{T}_f$ of the Riemann surface $\mathbb{C}\setminus P_f$ is the set of quasiconformal homeomorphisms of $\mathbb{C}\setminus P_f$ modulo post-composition with an affine map and isotopy relative $P_f$.
\end{defn}

The map
$$\sigma:[\varphi]\in\mathcal{T}_f\mapsto[\tilde{\varphi}]\in\mathcal{T}_f$$
is defined in the standard way via the Thurston diagram.

\begin{center}
	\begin{tikzcd}
		\mathbb{C},P_f \arrow[r, "{\tilde{\varphi}}"] \arrow[d, "f=c\circ f_0"]	& \mathbb{C},\tilde{\varphi}(P_f) \arrow[d, "g"] \\
		\mathbb{C},P_f \arrow[r, "{\varphi}"] & \mathbb{C},\varphi(P_f)
	\end{tikzcd}
\end{center}
\vspace{0.5cm}

More precisely, for every \qc\ map $\varphi:\mathbb{C}\to\mathbb{C}$ there exists another \qc\ map $\tilde{\varphi}:\mathbb{C}\to\mathbb{C}$ such that $g=\varphi\circ f\circ\tilde{\varphi}^{-1}$ is entire function. We define $\sigma[\varphi]:=[\tilde{\varphi}]$. For more details and the proof that this setup is well defined see \cite{IDTT1}. 

As in \cite{IDTT2}, we need to consider the subset of asymptotically conformal points of $\mathcal{T}_f$ and to use the fact that the $\sigma$-map is strictly contracting on this subset.

\begin{defn}[Asymptotically conformal points \cite{Gardiner}]
	\label{defn:as_conformal}
	A point $[\varphi]\in\mathcal{T}_f$ is called \emph{asymptotically conformal} if for every $\epsilon>0$ there is a compact set $C\subset\mathbb{C}\setminus P_f$ and a representative $\psi\in[\varphi]$ such that $\abs{\mu_\psi}<\epsilon$ a.e.\ on $(\mathbb{C}\setminus P_f)\setminus C$.
\end{defn}

\subsection{$\Id$-type maps}
\label{subsec:id_type}

One of the key notions in our constructions are $\id$-type maps. 

Let $f=c\circ f_0$ be the quasiregular function defined in Section~\ref{subsec:setup_and_contraction} where $f_0=p\circ\exp$. The union of post-singular ray tails $S_0=\cup_{i,j} R_{ij}$ is called the \emph{standard spider} of $f$.

\begin{defn}[$\Id$-type maps \cite{IDTT2}]
	\label{defn:id_type}
	A quasiconformal map $\varphi:\mathbb{C}\to \mathbb{C}$ is \idt\ if there is an isotopy $\varphi_u:\mathbb{C}\to\mathbb{C},\ u\in [0,1]$ such that ${\varphi_1=\varphi},\ \varphi_0=\id$ and $\abs{\varphi_u(z)-z}\to 0$ uniformly in $u$ as $S_0\ni z\to \infty$.
\end{defn}

We also say that $[\varphi]\in\mathcal{T}_f$ is \idt\ if $[\varphi]$ contains an $\id$-type map. Further, $\psi_u, u\in[0,1]$ is an \emph{isotopy \idt\ maps} if $\psi_u$ is an isotopy through maps \idt\ such that $\abs{\psi_u(z)-z}\to 0$ uniformly in $u$ as $S_0\ni z\to \infty$.

The following theorem says that the $\sigma$-map is invariant on the subset of $\id$-type points in $\mathcal{T}_f$.

\begin{thm}[Invariance of $\id$-type points \cite{IDTT2}]
	\label{thm:pullback_of_id_type}
	If $[\varphi]$ is \idt, then $\sigma[\varphi]$ is \idt\ as well.
	
	More precisely, if $\varphi$ is \idt, then there is a unique $\id$-type map $\hat{\varphi}$ such that $\varphi\circ f\circ\hat{\varphi}^{-1}$ is entire.
	
	Moreover, if $\varphi_u$ is an isotopy of $\id$-type maps, then the functions $\varphi_u\circ f\circ\hat{\varphi}_u^{-1}$ have the form $p_u\circ\exp$ where $p_u$ is a monic polynomial with coefficients depending continuously on $u$.
\end{thm}

\begin{remark}
	We keep using the hat-notation from Theorem~\ref{thm:pullback_of_id_type}: $\hat{\varphi}$ denotes the unique $\id$-type map so that $\varphi\circ f\circ\hat{\varphi}^{-1}$ is entire.	
\end{remark}

\subsection{Spiders}

Another important tool in the construction of the invariant compact subset of $\mathcal{T}_f$ are infinite-legged spiders. However, in this article we only need to prove one rather simple fact about them (Lemma~\ref{prp:constant_homotopy_near_infinity}); all other applications of spiders are already involved in \cite{IDTT2}.

\begin{defn}[Spider]
	\label{defn:spider}
	An image $S_{\varphi}$ of the standard spider $S_0$ under an $\id$-type map $\varphi$ is called a \emph{spider}.	
\end{defn}

Naturally, the image of a ray tail $R_{ij}$ under a spider map is called a \emph{leg}.

Let $V=\{w_n\}\subset\mathbb{C}$ be a finite set. Further, let $\gamma:[0,\infty]\to\hat{\mathbb{C}}$ be a curve such that $\gamma(0)\in V$, $\gamma(\infty)=\infty$, $\gamma|_{\mathbb{R}^+}\subset\mathbb{C}\setminus V$ and $\Re\gamma(t)\to+\infty$ as $t\to\infty$. On the set of all such pairs $(V,\gamma)$ one considers a map
$$W: (V,\gamma)\mapsto W(V,\gamma)\in F(V),$$
where $F(V)$ is the free group on $V$. This map uniquely encodes the homotopy type of $\gamma$ (with fixed endpoints) in $\mathbb{H}_r\setminus V$ where $\mathbb{H}_r=\{z\in\mathbb{C}:\Re z>r\}$ contains $V$ and $\gamma$ (a particular value of $r$ is irrelevant). Roughly speaking, we homotope $\gamma$ into a concatenation of a horizontal straight ray from $\gamma(0)$ to $+\infty$, and a loop with the base point at $\infty$, and by $W(V,\gamma)$ denote the representation of the loop via ``straight horizontal'' generators of the fundamental group of $(\mathbb{H}_r\setminus V)\cup\infty$. For the details of the construction we refer the reader to \cite[Subsection~6.1]{IDTT1}. We use this type of information for every leg of a spider to obtain a description of the $\id$-type points.

For every pair of $i\in \{1,2,...,m\},j\geq 0$ denote
$$\mathcal{O}_{ij}:=\{a_{kl}\in P_f: l<j\text{ or } l=j,k\leq i\}.$$

\begin{defn}[Leg homotopy word]
	Let $S_\varphi$ be a spider. Then the \emph{leg homotopy word} of a leg $\varphi(R_{ij})$ is 
	$$W_{ij}^\varphi:=W(\varphi(\mathcal{O}_{ij}),\varphi(R_{ij})).$$
\end{defn}

Since $W_{ij}^\varphi$ is an element of a free group, by $\abs{W_{ij}^\varphi}$ we naturally denote the length of the word representing this element.

\section{Preliminary constructions}
\label{sec:prel_constructions}

We consider the \qc\ function $f=c\circ f_0$ with $m$ singular orbits $\mathcal{O}_i=\{a_{ij}\}$ constructed in Subsection~\ref{subsec:setup_and_contraction} with its associated $\sigma$-map. We denote by $t_{ij}$ the potential of $a_{ij}$ and by $s_{ij}$ the first number in the external address of $a_{ij}$ (recall that the imaginary part of $R_{ij}$ tends to $2\pi s_{ij}/d$). Denote by $\SV(f)$ the set of finite singular values of $f$, i.e.\ the image of (finite) singular values of $f_0$ under the capture $c$. In this subsection we prove a few preliminary lemmas needed for the construction of the compact invariant subset $\mathcal{C}_f$. Note that we provide only the statements that are needed to upgrade the construction in \cite{IDTT2} to the degenerate case and mainly avoid duplicating results from \cite{IDTT2}. 

Let $\mathcal{P}=\{t_i\}_{i=1}^\infty$ be the set of all potentials of points in $P_f$ ordered so that $t_i<t_{i+1}$. Further, we define a set $$\mathcal{P}':=\{\rho_i:\rho_i=\frac{t_i+t_{i+1}}{2}\}.$$

The next lemma separates $P_f$ and $S_0$ into two parts: one near $\infty$ with good estimates and ``straight'' spider legs, and one near the origin where the points are located somewhat more chaotically. This is a statement \emph{solely} about $f_0$ and its standard spider.

\begin{lmm}[Good estimates for $S_0$ near $\infty$ \cite{IDTT2}]
	\label{lmm:good_behavior_of_f_0}
	There exists $t'>0$ such that:
	\begin{enumerate}
		\item if $t_i>t_j>t'$, then $t_i-t_j>2$,
		\item if $t_{kl}>t_{ij}>t'$, then $\abs{a_{kl}}>\abs{a_{ij}}+2$,
		\item if $\rho\in \mathcal{P}'$ is bigger than $t'$, then all $a_{ij}$ with potential less that $\rho$ are contained in $\mathbb{D}_{\rho-1}(0)$, while all $a_{ij}\in P_f$ with potential bigger than $\rho$ are contained in the complement of $\mathbb{D}_{\rho+1}(0)$.
	\end{enumerate}	
\end{lmm}

For all $\rho\in\mathcal{P}'$ we use the following notation. Denote $D_\rho:=\mathbb{D}_\rho (0)$, and for $i\in\{1,...,m\}$ let $N_i=N_i(\rho)$ be the maximal $j$-index of the points $a_{ij}$ contained in $D_\rho$. For $\rho>t'$ the disk $D_\rho$ contains first $N_i+1$ points $\{a_{i0},...,a_{iN_i}\}$ of $\mathcal{O}_i$, whereas the other points of $\mathcal{O}_i$ are in $\mathbb{C}\setminus{D}_{\rho}$.

\subsection{Clusters}

First of all, we define the central object of this article.
\begin{defn}[Cluster \cite{IDTT2}]
	\label{defn:cluster}
	We say that $a_{ij}$ and $a_{kl}$ are in the same \emph{cluster $Cl(t,s)$} if they have the same potential $t=t_{ij}=t_{kl}$ and the same first entry $s=s_{ij}=s_{kl}$ in the external address.
	
	A cluster is called non-trivial if it contains more than on point of $P_f$.
\end{defn}

Note that the set $P_f$ is a disjoint union of clusters and each cluster contains at most $m$ points. As can be seen from the asymptotic formula~\ref{eqn:as_formula}, for points $a_{ij}$ and $a_{kl}$ with big potentials being in the same cluster implies that distance between them is very small, whereas the distance between any pair of clusters is bounded from below.

The main reason, why we have to take clusters into account, is that if we have infinitely many non-trivial clusters, we cannot use the same model for $\mathcal{C}_f$ as in the non-degenerate case \cite{IDTT2}: we cannot separate points in the same clusters by disjoint discs that stable under the Thurston iteration (see discussion after \cite[Remark~5.10]{IDTT2}). On the other hand, we can control the behavior of marked points in clusters in a sense that under the Thurston iteration this position is basically getting scaled by a real constant (see Proposition~\ref{prp:good_big_disk_around_origin_cluster_estimates}).

Now, we are going to define two types of sets that will help us to describe the behavior of points inside of a cluster under the Thurston iteration. Their definitions might seem redundantly technical for the moment but this will pay out by simpler computations later.

\begin{defn}[$A_{x,n}$] 
	\label{defn:Axn}
	For $x>0$ and integer $n\geq 0$ let $A_{x,n}$ be the set of all complex numbers $\alpha$ such that
	\begin{enumerate}
		\item $\abs{\log\abs{\alpha}}<e^{-x/3}+...+e^{-F^n(x)/3}$,
		\item $\abs{\arg \alpha}<e^{-x/3}+...+e^{-F^n(x)/3}$.
	\end{enumerate}
\end{defn}

Visually $A_{x,n}$ is an annular sector containing $1$, see Figure~\ref{pic:Axn}. If $x$ increases, due to the fast convergence of the series $\sum_n e^{-F^n(x)/3}$, the sets $A_{x,n}$ shrink to arbitrarily small neighborhoods of $1$ uniformly in $n$.

The next lemma is the direct corollary of Definition~\ref{defn:Axn}. It presents a property that allows us to reduce the amount of computations later on.

\begin{lmm}[Pseudo-Multiplicativity of $A_{x,n}$]
	If $\alpha_i\in A_{F^{i-1}(x),0}$ for $i\in\{1,...,k\}$, then $\alpha_1\cdot...\cdot\alpha_k\in A_{x,k-1}$.
\end{lmm}

\begin{figure}[h]
	\includegraphics[width=\textwidth]{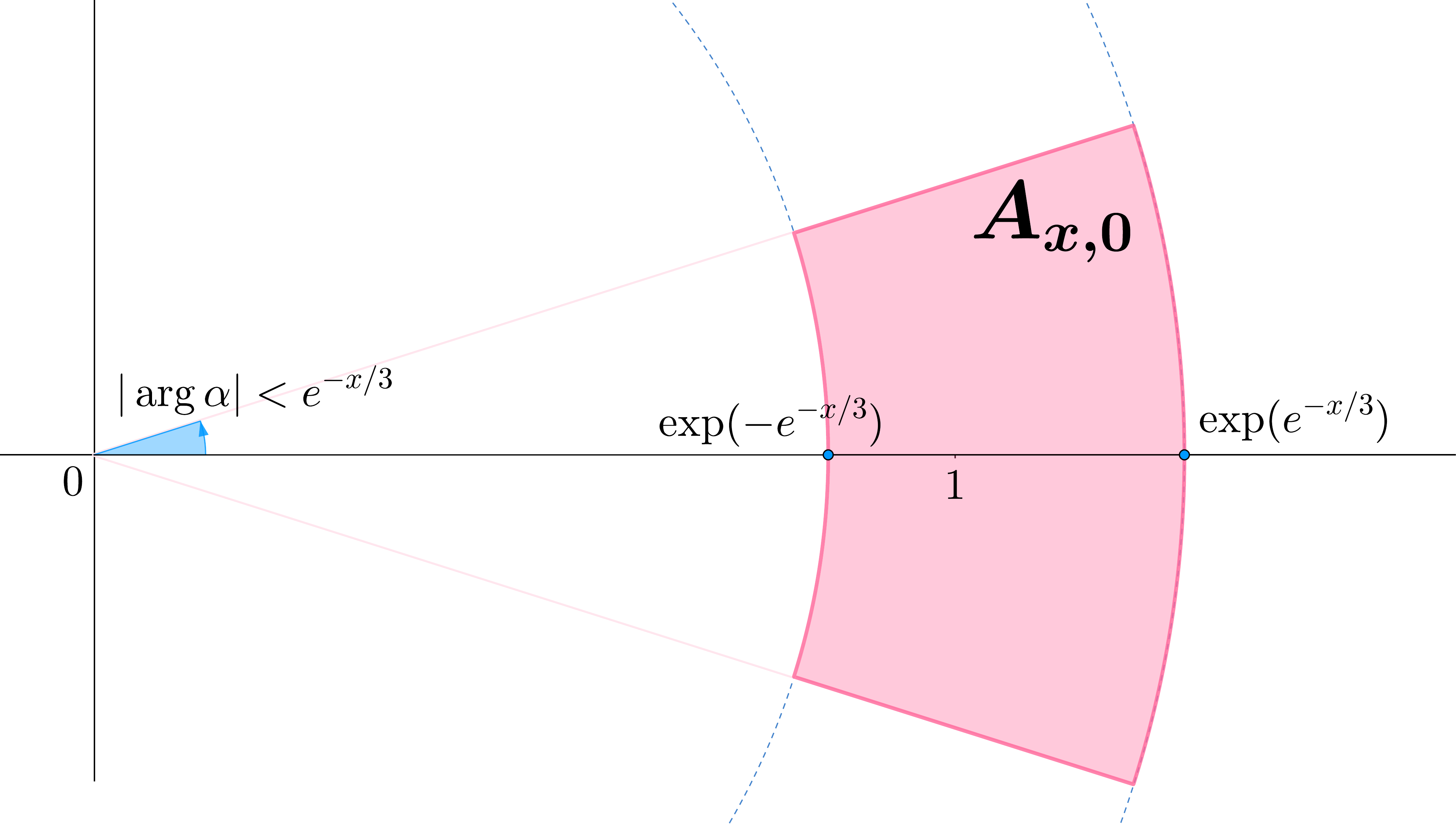}
	\caption{Schematic example of $A_{x,0}$.}
	\label{pic:Axn}
\end{figure}

\begin{figure}[h]
	\includegraphics[width=0.8\textwidth]{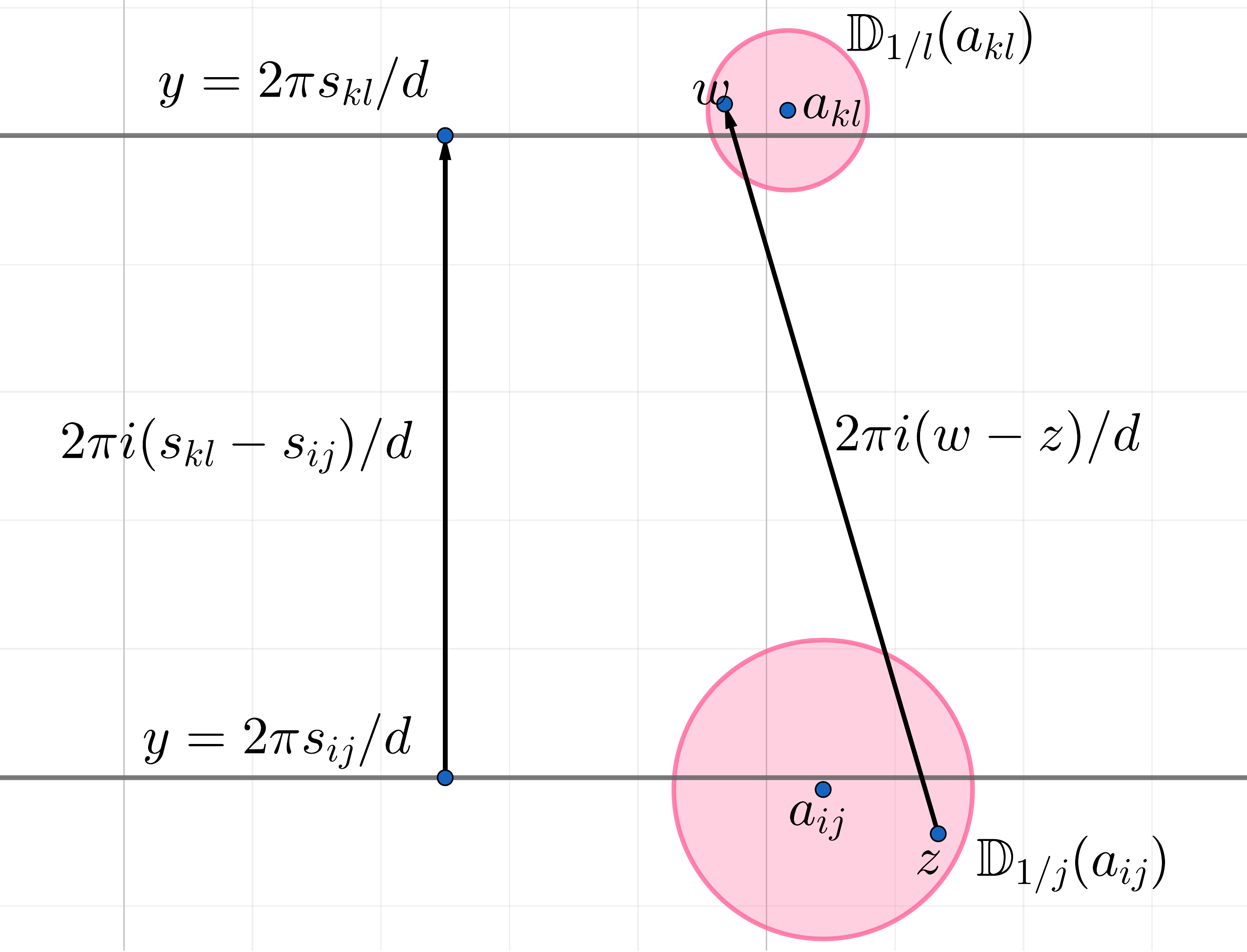}
	\caption{Schematic construction of $D_{ij}^{kl}$.}
	\label{pic:Bijkl}
\end{figure}

The next notion helps to characterize the difference between marked point that have the same potential but are in different clusters.
\begin{defn}[$D_{ij}^{kl}$]
	\label{defn:Dijkl}
	For $a_{ij}$ and $a_{kl}$ having the same potential  but belonging to different clusters define 
	$$D_{ij}^{kl}:=\left\{\frac{d(w-z)}{2\pi i(s_{kl}-s_{ij})}: w\in\mathbb{D}_{1/l}(a_{kl}), z\in\mathbb{D}_{1/j}(a_{ij})\right\}.$$	
\end{defn}

In other words, if we surround $a_{kl}$ and $a_{ij}$ by the disks of radii $1/l$ and $1/j$, then $D_{ij}^{kl}$ is just the difference of these two balls divided over the normalizing coefficient $2\pi i(s_{kl}-s_{ij})/d$. This is illustrated on Figure~\ref{pic:Bijkl}.

It is easy to see from Definition~\ref{defn:Dijkl} that $D_{ij}^{kl}$ is just a disk.

\begin{lmm}
	\label{lmm:Dijkl_is_disk}
	$D_{ij}^{kl}=\mathbb{D}_{r_{ij}^{kl}}\left(\frac{d(a_{kl}-a_{ij})}{2\pi i(s_{kl}-s_{ij})}\right)$ where $r_{ij}^{kl}=\frac{d(1/l+1/j)}{2\pi\abs{s_{kl}-s_{ij}}}$.
	
	If $j$ and $l$ are big enough, then
	\begin{enumerate}
		\item $1\in D_{ij}^{kl}$,
		\item $D_{ij}^{kl}$ is contained in arbitrarily small neighborhoods of $1$.
	\end{enumerate}
\end{lmm}
\begin{proof}
	That $D_{ij}^{kl}=\mathbb{D}_{r_{ij}^{kl}}\left(\frac{d(a_{kl}-a_{ij})}{2\pi i(s_{kl}-s_{ij})}\right)$ is a direct consequence of Definition~\ref{defn:Dijkl}.
	
	\textbf{(1)} It is implied by Definition~\ref{defn:Dijkl} that the potentials of $a_{ij}$ and $a_{kl}$ are equal: $t_{ij}=t_{kl}$. Hence from the asymptotic formula follows that
	$$\frac{d(a_{kl}-a_{ij})}{2\pi i(s_{kl}-s_{ij})}=\frac{(s_{kl}-s_{ij})+O(e^{-t_{ij}/2})}{(s_{kl}-s_{ij})}=1+\frac{O(e^{-t_{ij}/2})}{\abs{s_{kl}-s_{ij}}}.$$
	
	The second summand is smaller than $r_{ij}^{kl}$ for big $j,l$. The claim follows.
	
	\textbf{(2)} Since $a_{ij}$ and $a_{kl}$ belong to different clusters, we have $\abs{s_{kl}-s_{ij}}\geq 1$. Hence $r_{ij}^{kl}\to 0$ as $j,l\to \infty$.
\end{proof}

Next lemma provides an estimate on how the Thurston iteration changes the relative position of marked points inside of the clusters with sufficiently high potentials: for points $a_{ij},a_{kl}$ belonging to the same cluster with potential $t$, the number $\hat{\varphi}(a_{kl})-\hat{\varphi}(a_{ij})$ is equal to $\varphi_u(a_{k(l+1)})-\varphi_u(a_{i(j+1)})$ scaled by $F'(t)$ and multiplied by a correcting coefficient $\alpha\in A_{t,0}$ (which is getting arbitrarily close to $1$ as $t$ grows).

\begin{prp}[Negligible rotation within clusters]
	\label{prp:good_big_disk_around_origin_cluster_estimates}
	For all $\rho\in\mathcal{P}'$ big enough holds the following statement.
	
	If $a_{ij},a_{kl}\notin{D}_\rho$ belong to the same cluster $Cl(t,s)$, and $\varphi_u, u\in[0,1]$ is an isotopy \idt\ maps such that $\varphi_0=c^{-1}$ and for all $u\in[0,1]$:
	\begin{enumerate}
		\item $\varphi_u(\SV(f))\subset D_\rho$,
		\item $\abs{\varphi_u(a_{i(j+1)})-a_{i(j+1)}}<1/(j+1)$,
		\item $\abs{\varphi_u(a_{k(l+1)})-a_{k(l+1)}}<1/(l+1)$,
	\end{enumerate}
	then for every $u\in[0,1]$ holds
	$$\hat{\varphi}_u(a_{kl})-\hat{\varphi}_u(a_{ij})=\frac{\varphi_u(a_{k(l+1)})-\varphi_u(a_{i(j+1)})}{F'(t)}\alpha_u,$$
	where $\alpha_u\in A_{t,0}$.
\end{prp}
\begin{proof}
	Note that the statement of the theorem is vacuous if $d=1$: we do not have clusters containing more than one point if we only have one singular orbit.
	
	Let $g_u=\varphi_u\circ f\circ\hat{\varphi}_u^{-1}=q_u\circ \exp$ with $q_u(z)=z^d+b_{d-1}^u z^{d-1}+...+b_1^u z+b_0^u$ and $g_0=f_0$
	
	From \cite[Proposition~5.3]{IDTT2} we know that for $\rho$ big enough and every $a_{ij}\notin D_\rho$ holds $$\abs{\hat{\varphi}_u(a_{ij})-a_{ij}}<\frac{1}{j}.$$ On the other hand, from the formula~\ref{eqn:as_formula} we know that $\Im a_{ij}-2\pi s_{ij}\to 0$ as $j\to\infty$. Hence $\Im\hat{\varphi}_u(a_{ij})-2\pi s_{ij}\to 0$ as $j\to\infty$.
	
	Now let $a_{ij},a_{kl}\notin D_\rho$ belong to the same cluster $Cl(t,s)$, and $\gamma_u$ is the straight line segment joining $\hat{\varphi}_u(a_{ij})$ and $\hat{\varphi}_u(a_{kl})$. For all $\rho$ big enough we have	
	$$\varphi_{u}(a_{kl})-\varphi_{u}(a_{ij})=g_u(\hat{\varphi}_u(a_{kl}))-g_u(\hat{\varphi}_u(a_{ij}))=$$
	$$\int\limits_{\gamma_u}g_u'(z)dz=F'(t)\int\limits_{\gamma_u}\frac{g_u'(z)}{F'(t)}dz$$
	
	From Lemma~\ref{lmm:bounds_of_coefficients_through_SV} we know that $\abs{b_k^u}<L\rho^\frac{d-k}{d}$. Since $t/\rho\to\infty$ as $\rho\to\infty$ and $d>1$, the last expression is equal to	
	$$F'(t)\int\limits_{\gamma_u}\left(e^{d\left(z-t\right)}+O\left(e^{-t/2}\right)\right)dz$$
	
	In the proof of the item (3) of \cite[Proposition~5.3]{IDTT2} we obtained the estimate $\hat{\varphi}_u(a_{ij})-a_{ij}=O\left(e^{-t/2}\right)$, and from the asymptotic formula~\ref{eqn:as_formula} we know that $a_{ij}=t+2\pi i s/d+ O(e^{-t/2})$. Hence, $\hat{\varphi}_u(a_{ij})=t+2\pi i s/d+O\left(e^{-t/2}\right)$. Analogous formula holds for $a_{kl}$: $\hat{\varphi}_u(a_{kl})=t+2\pi i s/d+O\left(e^{-t/2}\right)$. Consequently, for every $z\in\gamma_u$ we also have $z=t+2\pi i s/d+O\left(e^{-t/2}\right)$. Thus,	
	$$F'(t)\int\limits_{\gamma}\left(e^{d\left(z-t\right)}+O\left(e^{-\frac{t}{2}}\right)\right)dz=F'(t)\int\limits_{\gamma}\left(e^{O\left(\exp \left(-\frac{t}{2}\right)\right)}+O\left(e^{-\frac{t}{2}}\right)\right)dz=$$	
	$$F'(t)\int\limits_{\gamma}\left(1+O\left(e^ {-\frac{t}{2}}\right)\right)dz=F'\left(t\right)\left(\hat{\varphi}_u\left(a_{kl}\right)-\hat{\varphi}_u\left(a_{ij}\right)\right)\left(1+O\left(e^ {-\frac{t}{2}}\right)\right)=$$
	$$\frac{F'(t)}{1+O\left(e^ {-\frac{t}{2}}\right)}\left(\hat{\varphi}_u(a_{kl})-\hat{\varphi}_u(a_{ij})\right).$$
	
	Let $\alpha_u=1+O\left(e^ {-t/2}\right)$ be the denominator in the last expression, and let $\abs{O\left(e^ {-\frac{t}{2}}\right)}<Ce^ {-\frac{t}{2}}$ for some $C>0$. Choose $\rho$ big enough so that 
	$$2Ce^ {-\frac{t}{2}}<e^{-\frac{t}{3}}<1/2.$$
	Then 
	$$1-e^{-\frac{t}{3}}<\abs{\alpha_u}< 1+e^{-\frac{t}{3}}.$$
	After taking logarithm we obtain $\abs{\log\abs{\alpha_u}}<e^{-t/3}$. Also we have
	$$\abs{\sin\arg\alpha_u}<\frac{Ce^ {-t/2}}{1-Ce^ {-t/2}}<2Ce^ {-t/2}.$$
	Hence, 
	$$-e^{-t/3}<\arg\alpha_u<e^{-t/3}.$$
	Thus $\alpha_u\in A_{t,0}$.
\end{proof}

We also need another piece of notation.

\begin{defn}[$H,L$]
	For every pair $a_{ij},a_{kl}$ belonging to the same cluster define
	\begin{itemize}
		\item $H=H(a_{ij},a_{kl})$ to be the smallest positive integer such that $a_{i(j+H)}$ and $a_{k(l+H)}$ belong to different clusters,
		\item $L=L(a_{ij},a_{kl})$ to be the smallest positive integer such that $a_{i(j-L)}$ and $a_{k(l-L)}$ are defined and belong to different clusters, if there is no such integer we define $L:=\infty$.
	\end{itemize}	
\end{defn}

Note that $H(a_{ij},a_{kl})$ is always less than infinity: otherwise this would mean that the points $a_{ij},a_{kl}$ had the same potential and external address. From Theorem~\ref{thm:as_formula} it then follows $a_{ij}=a_{kl}$.

Now, we prove an easy corollary of Proposition~\ref{prp:good_big_disk_around_origin_cluster_estimates} which provides cluster estimates for the singular orbits of $f$. Roughly speaking, it says that if two points are in the same cluster, we can estimate the distance between them if we know after how many iterations they land in different clusters. Basically, this is a scaling by the real constant $d(F^{\circ H}(t))'$ multiplied by correcting coefficients $\nu$ and $\delta$.

\begin{lmm}[Cluster estimates for $P_f$]
	\label{lmm:cluster_estimates_P_f}
	For all $j$ big enough holds the following statement.
	
	If $a_{ij}$ and $a_{kl}$ are in the same cluster, then
	\begin{equation}
		\label{eqn:cluster_estimates}
		a_{kl}-a_{ij}=\frac{2\pi i(s_{k(l+H)}-s_{i(j+H)})}{d(F^{\circ H}(t))'}\nu\delta,
	\end{equation}
	where $t$ is the potential of $a_{ij}$ and $a_{kl}$, $H=H(a_{ij},a_{kl})$, $\nu=\nu(i,j,k,l)\in A_{t,H-1}$ and $\delta=\delta(i,j,k,l)\in D_{i(j+H)}^{k(l+H)}$.	
\end{lmm}
\begin{proof}
	One simply needs to note that
	$$a_{k(l+H)}-a_{i(j+H)}=\frac{2\pi i(s_{k(l+H)}-s_{i(j+H)})}{d}\delta$$
	where $\delta\in D_{i(j+H)}^{k(l+H)}$,
	and to apply $H$ times the Proposition~\ref{prp:good_big_disk_around_origin_cluster_estimates} with $\varphi_u\equiv c^{-1}$, obtaining at the step $n$ the estimates for $a_{k(l+H-n)}-a_{i(j+H-n)}$.
\end{proof}

We also need to introduce special coefficients $\beta(a_{ij},a_{kl})$, in order to estimate relative distances between marked points outside of $D_\rho$.

\begin{defn}[$\beta$]
	Let $a_{ij},a_{kl}\in P_f$.	Denote
	$$\beta(a_{ij},a_{kl}):=
	\begin{cases}
		\frac{\pi(s_{k(l+H)}-s_{i(j+H)})}{d(F^{\circ H}(t))'} & \mbox{if\,} a_{ij},a_{kl} \mbox{\,are in the same cluster $Cl(t,s)$;}\\
		\pi/{2d}& \mbox{if\,} a_{ij},a_{kl} \mbox{\,are in different clusters;}
	\end{cases}$$	
	where $H=H(a_{ij},a_{kl})$.
\end{defn} 

From Lemma~\ref{lmm:cluster_estimates_P_f} and the asymptotic formula~\ref{eqn:as_formula} follows immediately that for all pairs $a_{ij}, a_{kl}$ such that $j$ is big enough holds
\begin{equation}
	\label{eqn:bigger_than_beta}
	\abs{a_{ij}-a_{kl}}>\beta(a_{ij},a_{kl});
\end{equation}
and for all pairs $a_{ij},a_{kl}$ belonging to the same cluster such that $j$ is big enough holds
\begin{equation}
	\label{eqn:smaller_than_beta}
	\abs{a_{ij}-a_{kl}}<4\beta(a_{ij},a_{kl}).
\end{equation}

\subsection{No combinatorics near $\infty$ (almost)}

The next proposition is analogous to \cite[Proposition~5.6]{IDTT2} and is supposed to provide estimates on the homotopy type of spider legs in presence of non-trivial clusters.

\begin{prp}[No combinatorics near $\infty$]
	\label{prp:constant_homotopy_near_infinity}
	There exists a constant $C>0$ such that the following statement holds for all $\rho\in\mathcal{P}'$ big enough.
	
	If $\varphi_u, u\in[0,1]$ is an isotopy of $\id$-type maps satisfying:
	\begin{enumerate}
		\item $\varphi_0|_{\mathbb{C}\setminus\mathbb{D}_\rho}=\id$,		
		\item $\varphi_u(a_{ij})\in D_\rho$ for $j\leq N_i$,
		\item $\abs{\varphi_u(a_{ij})-a_{ij}}<1/j$ for $j>N_i$,
		\item for all $a_{ij},a_{kl}$ with $j>N_i,l>N_k$, belonging to the same cluster holds
		$$\varphi_u(a_{kl})-\varphi_u(a_{ij})=\frac{2\pi i(s_{k(l+H)}-s_{i(j+H)})}{d(F^{\circ H}(t))'}\nu_u\delta_u$$
		where $t$ is the potential of $a_{ij}$ and $a_{kl}$, $H=H(a_{ij},a_{kl})$, $\nu_u=\nu_u(\varphi_u,i,j,k,l)\in A_{t,H-1}$ and $\delta_u=\delta_u(\varphi_u,i,j,k,l)\in D_{i(j+H)}^{k(l+H)}$,
	\end{enumerate}
	then for every $u$ and every $j>N_i$ we have $$\abs{W_{ij}^{\varphi_u}}<C.$$
\end{prp}
\begin{proof}
	The proof is almost the same as of \cite[Proposition~5.6]{IDTT2} except that we need to take care of possible non-trivial clusters.
		
	Note that if $j$ is big enough, than the points $a_{ij}\notin D_\rho$ belonging to different clusters move under $\varphi_u$ inside of mutually disjoint disks $\mathbb{D}_{1/j}(a_{ij})$, and mutual distance between such disks is bounded from below by $\pi/2d$. 
	
	On the other hand, since for big $t$ the sets $A_{t,H-1}$ and $D_{i(j+H)}^{k(l+H)}$ are contained in arbitrarily small neighborhoods of $1$, for all essentially big $t$ the marked points $\varphi_u(a_{ij})$ inside of a cluster with potential $t$ have different imaginary parts and this holds for all $u\in[0,1]$, i.e.\ the points inside of clusters with high enough potentials are vertically linearly ordered. 
	
	In \cite[Lemma~2.14]{IDTT2} we proved that for all big enough $j$ the ray tail $R_{ij}=\varphi_0(R_{ij})$ has the strictly increasing real part. Hence, for every $u\in[0,1]$ in the homotopy class (relative $P_f$) of $\varphi_u(R_{ij})$ there is a representative with the strictly increasing real part.
	
	Next, due to the asymptotic formula~\ref{eqn:as_formula} there is a universal constant $M>0$ such that for every $a_{ij}\in P_f$, every isotopy $\varphi_u$ satisfying conditions of the proposition, and every $u\in[0,1]$ there is at most $M$ points $a_{kl}\in \mathcal{O}_{ij}$ with $\Re\varphi_u(a_{kl})>\Re\varphi_u(a_{ij})-1$. Hence, from \cite[Lemma~6.6]{IDTT1} and \cite[Lemma~6.8]{IDTT1} we see that there exists a universal constant $C$ such that for all $\rho$ big enough and all $j>N_i$ we have
	$\abs{W_{ij}^{\varphi_u}}<C$.	
\end{proof} 

\subsection{Separation of preimages in $D_\rho$}

As in \cite{IDTT2}, we need a tool to ``separate'' the marked points inside of $D_\rho$. This tools is a simple ``division over the maximum of derivative''. However, as will be seen in the Theorem~\ref{thm:invariant_subset}, more control will be needed than in the non-degenerate case \cite{IDTT2}.

\begin{defn}[Maximum of the derivative \cite{IDTT2}]
	For $\rho=\frac{t_n+t_{n+1}}{2}\in\mathcal{P}'$, define $$M_\rho:=\sup\limits_{\substack{g\in\mathcal{N}_d \\ SV(g)\subset D_\rho}}\sup\limits_{\Re z<{(d+1) t_n}}\abs{g'(z)}.$$	
\end{defn}

For all $\rho$ big enough, $M_\rho$ satisfies the following inequalities:
\begin{equation}
	\label{eqn:M_rho}
	\rho<M_\rho<Ke^{d^3 t_n}
\end{equation}
for some $K>0$ depending only on $d$. The former inequality follows directly from the definition of $M_\rho$, and the latter from Lemma~\ref{lmm:bounds_of_coefficients_through_SV}.

Next proposition helps to estimate mutual distances between marked points inside of $D_\rho$. 

\begin{prp}[Separation of preimages \cite{IDTT2}]
	\label{prp:division_over_derivative}	
	For all $\rho\in\mathcal{P}'$ big enough holds the following statement.
	
	If $\varphi$ is an $\id$-type map such that $\varphi(\SV(f))\subset D_\rho$, and points $x,y\in\mathbb{C}\setminus\SV(f)$ are satisfy $$\varphi\circ f(x),\varphi\circ f(y)\in\overline{\mathbb{D}}_{\max\{|a_{i(N_i+1)}|+1\}}(0),$$ then
	$$\abs{\hat{\varphi}(x)-\hat{\varphi}(y)}\geq\frac{\abs{\varphi\circ f(x)-\varphi\circ f(y)}}{M_\rho}.$$
\end{prp} 

Finally, in the last lemma before the construction of the invariant compact subset $\mathcal{C}_f$ we prove certain estimates on mutual distances between points in $P_f$. Roughly speaking, the lemma gives estimates on how close the marked points may be close to each other based on the combinatorial information (with a special attention in (2) on the case when the orbit of the asymptotic value is involved).

Note that it can be in principle formulated without the ``$\rho$-notation'', but we prefer to use it for the sake of the more convenient application in Theorem~\ref{thm:invariant_subset}.

From now on we agree that $\mathcal{O}_1=\{a_{1j}\}_{j=0}^\infty$ is the orbit of the \emph{asymptotic} value of $f$, i.e.\ $a_{10}=c(\text{as.\ val.\ of }f_0)$.  

\begin{lmm}[Distance between points in $P_f$]
	\label{lmm:dist_points_P_f}
	For all $\rho\in\mathcal{P}'$ big enough hold the following two statements. 
	\begin{enumerate}
		\item If $j\leq N_i,l\leq N_k,ij\neq kl$, and $n=\min\{N_i+1-j, N_k+1-l\}$, then $$\abs{a_{kl}-a_{ij}}>\frac{\beta\left(a_{i(j+n)},a_{k(l+n)}\right)}{(M_\rho)^n}.$$
		\item If $a_{k(N_k+1)}$ and $a_{1(N_1+1)}$ are in the same cluster, $L=L(a_{k(N_k+1)},a_{1(N_1+1)})$, and $N=\max_{i\leq m}N_i$, then: 
		\begin{enumerate}
			\item if $0\leq n<L$, then
			$$\abs{a_{k(N_k+1-n)}-a_{1(N_1+1-n)}}<
			4\beta\left(a_{k(N_k+1)},a_{1{(N_1+1)}}\right)M_\rho^{2dNn},$$
			
			\item if $n\geq L$, then
			$$\abs{a_{k(N_k+1-n)}-a_{1(N_1+1-n)}}>\left(\frac{1}{M_\rho}\right)^{2d^4N+n-L}.$$
		\end{enumerate}
	\end{enumerate}
\end{lmm}
\begin{proof}
	\textbf{(1)} Applying Proposition~\ref{prp:division_over_derivative} with $\hat{\varphi}=c^{-1}$ to the pair $a_{i(j+n_1+1)},a_{k(l+n_1+1)}$ with $0\leq n_1\leq n$ we obtain
	$$\abs{a_{i(j+n_1)}-a_{k(l+n_1)}}>\frac{\abs{a_{i(j+n_1+1)}-a_{k(l+n_1+1)}}}{M_\rho}.$$
	At the same time due to formula~\ref{eqn:bigger_than_beta}
	$$\abs{a_{i(j+n)}-a_{k(l+n)}}>\beta\left(a_{i(j+n)},a_{k(l+n)}\right).$$
	The conclusion follows after multiplication of the inequalities above.
	
	\textbf{(2)} Without loss of generality assume that $L<N_1$, otherwise increase $\rho$. The proof of $(a)$ goes by induction. The case $n=0$ coincides with formula~\ref{eqn:smaller_than_beta}.
	
	Assume that $n+1<L$ and 
	$$\abs{a_{k(N_k+1-n)}-a_{1(N_1+1-n)}}<
	4\beta\left(a_{k(N_k+1)},a_{1{(N_1+1)}}\right)M_\rho^{2dNn}.$$
	
	Let $\gamma$ be the straight line segment joining $a_{k(N_k+1-n)}$ and $a_{1(N_1+1-n)}$. After increasing $\rho$ if needed, we might assume that the distance from $\gamma$ to the singular values of $f_0$ is bigger than $1/(M_\rho)^{N+2}$: the number $4\beta\left(a_{k(N_k+1)},a_{1{(N_1+1)}}\right)M_\rho^{2dNn}$ tends to $0$ as $\rho\to\infty$ since $a_{k(N_k+1)},a_{1(N_1+1)}$ are in the same cluster, which makes $\beta\left(a_{k(N_k+1)},a_{1{(N_1+1)}}\right)$ very small even comparing to $1/(M_\rho)^{2dNn}$.
	
	Further, without loss of generality we might assume that the preimage $\tilde{\gamma}$ of $\gamma$ under $f_0=p\circ\exp$ which starts at $a_{1(N_1-n)}$ ends at $a_{k(N_k-n)}$: this is true for all big enough $\rho$. Hence, it is enough to estimate the length of $\tilde{\gamma}$, in order to bound the distance between $a_{1(N_1-n)}$ and $a_{k(N_k+1-n)}$.
	
	If $z\in\tilde{\gamma}$, then the distance from $e^z$ to $0$ and the critical points of $p$ is bigger than the distance between $f_0(z)$ and the singular values of $f_0$ divided over $M_\rho$, i.e.\ bigger than $1/(M_\rho)^{N+3}$. Hence,
	$$\abs{f_0'(z)}=\abs{p'(e^z)e^z}>d\left(\frac{1}{M_\rho^{N+3}}\right)^d>\frac{1}{M_\rho^{2dN}}.$$
	
	Thus,
	$$\abs{a_{k(N_k-n)}-a_{1(N_1-n)}}\leq\frac{\abs{a_{k(N_k+1-n)}-a_{1(N_1+1-n)}}}{\min\limits_{z\in\gamma} \abs{f_0'(z)}}<$$
	$$\abs{a_{k(N_k+1-n)}-a_{1(N_1+1-n)}}M_\rho^{2dN}<4\beta\left(a_{k(N_k+1)},a_{1{(N_1+1)}}\right)M_\rho^{2dN(n+1)}.$$
	
	To prove $(b)$, first let $n=L$. From Lemma~\ref{lmm:distance_between_preimages_under_function} we know that the distance between different preimages of $a_{1(N_1+2-L)}$ is bigger than $(1/(M_\rho)^{N+2}\rho)^{d^4}>1/(M_\rho)^{d^4(N+3)}$. On the other hand, since $a_{1(N_1+1-L)}$ and $a_{k(N_k+1-L)}$ are in different clusters, exactly as in $(a)$ we obtain that 
	$$\abs{a_{k(N_k+1-L)}-\tilde{a}}<
	4\beta\left(a_{k(N_k+1)},a_{1{(N_1+1)}}\right)M_\rho^{2dNL}$$
	where $\tilde{a}$ is a preimage of $a_{1(N_1+2-L)}$ other than $a_{1(N_1+1-L)}$. But since for big $\rho$ the number $4\beta\left(a_{k(N_k+1)},a_{1{(N_1+1)}}\right)M_\rho^{2dNL}$ is much smaller than $1/(M_\rho)^{d^4(N+3)}$, we obtain
	$$\abs{a_{k(N_k+1-L)}-a_{1(N_1+1-L)}}>
	\frac{1}{2}\frac{1}{M_\rho^{d^4(N+3)}}>\frac{1}{M_\rho^{2d^4N}}.$$
	
	The inequality $(b)$ for $n>L$ is obtained by straightforward inductive application of Proposition~\ref{prp:division_over_derivative} as in the proof of $(a)$.
\end{proof}

\section{Compact invariant subset}
\label{sec:inv_compact_subset}

Now, we construct the compact invariant subset $\mathcal{C}_f$. First, we prove invariance. Recall that we agreed that $\mathcal{O}_1=\{a_{1j}\}_{j=0}^\infty$ is the orbit of the asymptotic value of $f$.

\begin{thm}[Invariant subset]
	\label{thm:invariant_subset}
	Let $f=c\circ\exp$ be the quasiregular function defined in Subsection~\ref{subsec:setup_and_contraction}. Further, let $\rho\in\mathcal{P}'$ and $\mathcal{C}_f(\rho)\subset\mathcal{T}_f$ be the \emph{closure} of the set of points in the \tei\ space $\mathcal{T}_f$ represented by $\id$-type maps $\varphi$ for which there exists an isotopy $\varphi_u, u\in[0,1]$ \idt\ maps such that $\varphi_0=\id$, $\varphi_1=\varphi$, and the following conditions are simultaneously satisfied.
	\begin{enumerate}		
		\item (Marked points stay inside of $D_\rho$) If $j\leq N_i$,
		$$\varphi_u(a_{ij})\in D_\rho.$$
		\item (Precise asymptotics outside of $D_\rho$) If $j>N_i$, then
		$$\abs{\varphi_u(a_{ij})-a_{ij}}<1/j.$$		
		\item (Separation inside of $D_\rho$) If $j\leq N_i,l\leq N_k,ij\neq kl$, and $n=\min\{N_i+1-j, N_k+1-l\}$, then $$\abs{\varphi_u(a_{kl})-\varphi_u(a_{ij})}>\frac{\beta\left(a_{i(j+n)},a_{k(l+n)}\right)}{(M_\rho)^n}.$$	
		\item (Bounded homotopy) If $j\leq N_i$, then 		
		$$\abs{W_{ij}^{\varphi_u}}<A^{N_i+1-j}\left(\frac{(N_i+1)!}{j!}\right)^4 C$$		
		where $A$ and $C$ are constants from \cite[Theorem~4.11]{IDTT2} and Proposition~\ref{prp:constant_homotopy_near_infinity}, respectively.
		\item (Rigidity in clusters) For all $a_{ij},a_{kl}$ with $j>N_i,l>N_k$ belonging to the same cluster, holds
		$$\varphi_u(a_{kl})-\varphi_u(a_{ij})=\frac{2\pi i(s_{k(l+H)}-s_{i(j+H)})}{d(F^{\circ H}(t))'}\nu_u\delta_u$$
		where $t$ is the potential of $a_{ij}$ and $a_{kl}$, $H=H(a_{ij},a_{kl})$, $\nu_u=\nu_u(\varphi_u,i,j,k,l)\in A_{t,H-1}$ and $\delta_u=\delta_u(\varphi_u,i,j,k,l)\in D_{i(j+H)}^{k(l+H)}$.
		\item (Clusters inside of $D_\rho$) If $a_{k(N_k+1)}$ and $a_{1(N_1+1)}$ belong to the same cluster, $L=L(a_{k(N_k+1)},a_{1(N_1+1)})$, and $N=\max_{i\leq m}N_i$, then:\\ 
		\begin{enumerate}
			\item if $0\leq n<L$, then
			$$\abs{\varphi_u(a_{k(N_k+1-n)})-\varphi_u(a_{1(N_1+1-n)})}<
			4\beta\left(a_{k(N_k+1)},a_{1{(N_1+1)}}\right)M_\rho^{2dNn},$$
			
			\item if $n\geq L$, then
			$$\abs{\varphi_u(a_{k(N_k+1-n)})-\varphi_u(a_{1(N_1+1-n)})}>\left(\frac{1}{M_\rho}\right)^{2d^4N+n-L}.$$
		\end{enumerate}
	\end{enumerate}
	
	Then if $\rho\in\mathcal{P}'$ is big enough, the subset $\mathcal{C}_f(\rho)$ is well-defined, invariant under the $\sigma$-map and contains $[\id]$.	
\end{thm}

Before the proof let us explain informally the meaning of conditions (1)-(6).

(1)-(2) say that the maps $\varphi$ are ``uniformly \idt'', that is, the marked points outside of a disk $D_\rho$ have precise asymptotics, while inside of $D_\rho$ we allow some more freedom.

(3) and (6b) say that the points inside of $D_\rho$ cannot come very close to each other. Moreover, it is required to control the distance to the asymptotic value --- if a marked point is too close to it, then after Thurston iteration its preimage has its real part close to $-\infty$, and this spoils condition (1).

(4) describes the homotopy information and provides bounds for leg homotopy words of legs with endpoints inside of $D_\rho$. Analogous bounds for marked points outside of $D_\rho$ are encoded \emph{implicitly} in conditions (2),(3) and (5).

(5) helps us to control the behavior of marked points within clusters. Note that we do not need it if there are no clusters --- in this case the marked points outside of $D_\rho$ are contained in small disjoint disks. If there are clusters, we cannot have such nice separation of orbits, but have a good control over the relative behavior of points. This condition prevents points within clusters to approach and rotate around each other.

(6) is needed for the same reason as (3) but in the very special case: when a marked point $a_{i(N_i+1)}$ is in the same cluster with the the marked point $a_{1(N_1+1)}$ on the orbit of the asymptotic value. In this case the estimate $\beta$ is very small. If we use the same estimates as in (3) to bound how the distance between their preimages changes under iteration of the $\sigma$-map, then we obtain a marked point that is too close to the asymptotic value --- after one more iteration its preimage will be ``spit'' towards $-\infty$. To avoid this we do the following trick. We make preimages of this pair of points stay ``very close'' as long as $a_{i(N_i+1-n)}$ and $a_{1(N_1+1-n)}$ are in the same cluster (condition (6a)). The first time they are in different clusters leads to a better estimate on the distance (condition (6b)).

\begin{proof}[Proof of Theorem~\ref{thm:invariant_subset}]
	A big part of the proof repeats the proof of \cite[Theorem~5.9]{IDTT2}. Whenever this is the case we avoid unnecessary repetitions and send the reader to \cite[Theorem~5.9]{IDTT2} for details.
	
 	Let $\mathcal{C}_f^\circ(\rho)\subset\mathcal{C}_f(\rho)$ be the set of points in $\mathcal{T}_f$ of which we take the closure in the statement of the theorem, that is, represented by $\id$-type maps $\varphi$ for which there exists an isotopy $\varphi_u, u\in[0,1]$ \idt\ maps such that $\varphi_0=\id$, $\varphi_1=\varphi$ and the conditions (1)-(4) are simultaneously satisfied. Since the $\sigma$-map is continuous, it is enough to prove invariance of $\mathcal{C}_f^\circ(\rho)$.
	
	First, we prove that for big $\rho$ the set $\mathcal{C}_f^\circ(\rho)$ contains $[c^{-1}]$: $c^{-1}$ can be joined to identity via the isotopy $c_u^{-1}$ where $c_u$ was constructed in Subsection~\ref{subsec:setup_and_contraction}, and this isotopy satisfies the conditions (1)-(6).
	
	That $c_u^{-1}$ satisfies (1)-(4) is discussed in \cite[Theorem~5.9]{IDTT2}, while (5) and (6) follow from Lemma~\ref{lmm:cluster_estimates_P_f} and Lemma~\ref{lmm:dist_points_P_f}, respectively.
	
	For all $\varphi\in\mathcal{C}_f^\circ(\rho)$ after concatenation with $c_u^{-1}$ we can obtain the isotopy $\psi_u$ \idt\ maps with $\psi_0=c^{-1},\psi_1=\varphi$ and satisfying conditions (1)-(6). Then $\hat{\psi}_u$ is an isotopy \idt\ maps with $\hat{\psi}_1=\id$. Let $g_u(z)=\psi_u\circ f\circ\hat{\psi}_u^{-1}(z)=p_u\circ\exp(z)$. Now, we prove that $\hat{\psi}_u$ satisfies each of the items (1)-(6): it will follow that $\hat{\varphi}\in\mathcal{C}_f^\circ(\rho)$. 
	
	We are going to prove that each of the conditions (1)-(6) for $\hat{\psi}$ follows from the conditions (1)-(6) for $\psi$.
	
	\textbf{(4)} Literally repeats the proof of (4) in \cite[Theorem~5.9]{IDTT2} using Proposition~\ref{prp:constant_homotopy_near_infinity} instead of \cite[Proposition~5.6]{IDTT2}.
	
	\textbf{(3)} Follows directly from Proposition~\ref{prp:division_over_derivative}.
	
	\textbf{(2)} The proof fully repeats the proof of \cite[Theorem~5.9]{IDTT2}(2).
		
	\textbf{(1)} We show that $\pm\rho/2$ are upper and lower bounds for both the real and imaginary parts of $\hat{\psi}_u(a_{ij})$ for $a_{ij}\in D_\rho$: this will give the desired bound on $\abs{\hat{\psi}_u(a_{ij})}$. We will only show that $\Re\hat{\psi}_u(a_{ij})>-\rho/2$, the other bounds are obtained exactly as in \cite[Theorem~5.9]{IDTT2}.
	
	Let $a_{ij}$ be such that $j\leq N_i$ and $N=\max N_i$. If $j+1+N_1=N_i$ and $a_{1N_1},a_{iN_i}$ are in the same cluster, then from (6b)
	$$\abs{\varphi_u(a_{i(j+1)})-\varphi_u(a_{10})}>\left(\frac{1}{M_\rho}\right)^{2d^4N+N+1-L}.$$
	Otherwise for $n=\min\{N_i-j, N_1+1\}$ from (3) we have 
	$$\abs{\varphi_u(a_{i(j+1)})-\varphi_u(a_{10})}>\frac{\beta\left(a_{i(j+n)},a_{k(l+n)}\right)}{(M_\rho)^{N+1}}>\frac{\pi}{2d(M_\rho)^{N+1}}.$$
	In any case we might assume that the first inequality holds (this is true for big enough $\rho$). Then
	$$\abs{p_u^{-1}(\psi_u(a_{i(j+1)}))-p_u^{-1}(\psi_u(a_{10}))}=\abs{\exp(\hat{\psi}_u(a_{ij}))-0}>\left(\frac{1}{M_\rho}\right)^{2d^4N+N+1-L+1}.$$
	Hence,
	$$\Re\hat{\psi}_u(a_{ij})>-(N(2d^4+1)+2-L)\log M_\rho.$$
	If $\rho$ is big, using the inequality~\ref{eqn:M_rho} we see that the right hand side of the last expression is bigger than $-\rho/2$.
	
	\textbf{(5)} From Proposition~\ref{prp:good_big_disk_around_origin_cluster_estimates} we see that for all $j>N_i,l>N_k$, $a_{ij},a_{kl}$ belonging to the same cluster with potential $t$ $$\hat{\psi}_u(a_{kl})-\hat{\psi}_u(a_{ij})=\frac{\psi_u(a_{k(l+1)})-\psi_u(a_{i(j+1)})}{F'(t)}\alpha_u=$$ 
	$$\frac{2\pi i\left(s_{k(l+H)}-s_{i(j+H)}\right)}{d\left(F^{\circ (H-1)}\right)'\circ F(t)F'(t)}\left(\nu_u(\psi_u,i,j+1,k,l+1)\alpha_u\right)\delta_u(\psi_u,i,j+1,k,l+1)=$$
	$$\frac{2\pi i\left(s_{k(l+H)}-s_{i(j+H)}\right)}{d\left(F^{\circ H}(t)\right)'}\nu_u(\hat{\psi}_u,i,j,k,l)\delta_u(\hat{\psi}_u,i,j,k,l)$$ where $\alpha_u\in A_{t,0}$ and $H=H(a_{ij},a_{kl})$.	
	
	\textbf{(6)} The proof of (6) is goes analogously to the proof of Lemma~\ref{lmm:dist_points_P_f} (2).
	
	Without loss of generality assume that $L<N_1$, otherwise increase $\rho$. The case $n=0$ follows from (4).
	
	Assume that $n+1<L$ and 
	$$\abs{\psi_u(a_{k(N_k+1-n)})-\psi_u(a_{1(N_1+1-n)})}<
	4\beta\left(a_{k(N_k+1)},a_{1{(N_1+1)}}\right)M_\rho^{2dNn}.$$
	
	Let $\gamma_u$ be the straight line segment joining $\psi_u(a_{k(N_k+1-n)})$ and $\psi_u(a_{1(N_1+1-n)})$. After increasing $\rho$ if needed, assume that the distance from $\gamma_u$ to the singular values of $f_0$ is bigger than $1/(M_\rho)^{N+2}$: the number $4\beta\left(a_{k(N_k+1)},a_{1{(N_1+1)}}\right)M_\rho^{2dNn}$ tends to $0$ as $\rho\to\infty$ since $a_{k(N_k+1)},a_{1(N_1+1)}$ are in the same cluster. Hence, $\beta\left(a_{k(N_k+1)},a_{1{(N_1+1)}}\right)$ is very small even comparing to $1/(M_\rho)^{2dNn}$.
	
	Without loss of generality assume that the preimage $\tilde{\gamma}_u$ of $\gamma_u$ under $g_u=p_u\circ\exp$ which starts at $\hat{\psi}_u(a_{1(N_1-n)})$ ends at $\hat{\psi}_u(a_{k(N_k-n)})$: this is true since this holds for $\psi_0=c^{-1}$ and under the isotopy $\psi_u$ the singular values of $g_u$ do not cross $\gamma_u$. Hence, it is enough to estimate the length of $\tilde{\gamma}_u$, in order to bound the distance between $\hat{\psi}_u(a_{1(N_1-n)})$ and $\hat{\psi}_u(a_{k(N_k+1-n)})$.
	
	If $z\in\tilde{\gamma}_u$, then the distance from $e^z$ to $0$ and the critical points of $p_u$ is bigger than $1/(M_\rho)^{N+3}$. Hence,
	$$\abs{g_u'(z)}=\abs{p_u'(e^z)e^z}>d\left(\frac{1}{M_\rho^{N+3}}\right)^d>\frac{1}{M_\rho^{2dN}}.$$
	
	Thus, as in Lemma~\ref{lmm:dist_points_P_f} (2),
	$$\abs{\hat{\psi}_u(a_{k(N_k-n)})-\hat{\psi}_u(a_{1(N_1-n)})}\leq\frac{\abs{\psi_u(a_{k(N_k+1-n)})-\psi_u(a_{1(N_1+1-n)})}}{\min\limits_{z\in\gamma} \abs{g_u'(z)}}<$$
	$$\abs{\psi_u(a_{k(N_k+1-n)})-\psi_u(a_{1(N_1+1-n)})}M_\rho^{2dN}<4\beta\left(a_{k(N_k+1)},a_{1{(N_1+1)}}\right)M_\rho^{2dN(n+1)}.$$
	
	To prove $(b)$, first consider $n=L$. From Lemma~\ref{lmm:distance_between_preimages_under_function} as in Lemma~\ref{lmm:dist_points_P_f} (2) follows that the distance between different preimages of $\psi_u(a_{1(N_1+2-L)})$ under $g_u$ is bigger than $1/(M_\rho)^{d^4(N+3)}$. On the other hand, since $a_{1(N_1+1-L)}$ and $a_{k(N_k+1-L)}$ are in different clusters, as in $(a)$ we have 
	$$\abs{\hat{\psi}_u(a_{k(N_k+1-L)})-\tilde{a}}<
	4\beta\left(a_{k(N_k+1)},a_{1{(N_1+1)}}\right)M_\rho^{2dNL}$$
	where $\tilde{a}$ is a preimage of $\psi_u(a_{1(N_1+2-L)})$ under $g_u$ other than $\hat{\psi}_u(a_{1(N_1+1-L)})$. Since for big $\rho$ the number $4\beta\left(a_{k(N_k+1)},a_{1{(N_1+1)}}\right)M_\rho^{2dNL}$ is much smaller than $1/(M_\rho)^{d^4(N+3)}$, we have
	$$\abs{\hat{\psi}_u(a_{k(N_k+1-L)})-\hat{\psi}_u(a_{1(N_1+1-L)})}>
	\frac{1}{2}\frac{1}{M_\rho^{d^4(N+3)}}>\frac{1}{M_\rho^{2d^4N}}.$$
	
	The inequality $(6b)$ for $n>L$ can be obtained by application of Proposition~\ref{prp:division_over_derivative} as in the proof of $(3)$.
\end{proof}

\begin{remark}
	Every point of $\mathcal{C}_f$ is evidently asymptotically conformal. Indeed, near infinity the marked points are separated into clusters that move inside of disjoint small disks, whereas there is only a negligible relative movement of marked points under isotopies inside of every cluster.
\end{remark}

\begin{remark}
	Note that Theorem~\ref{thm:invariant_subset} defines an invariant set also in the case when $\underline{s}_i$ is (pre-)periodic. Indeed, absence of (pre-)periodicity is not used in the proof. But we do not have enough tools to prove that in the (pre-)periodic case the set $\mathcal{C}_f(\rho)$ would be compact. For this one needs to upgrade spider construction to allow each leg contain infinitely many marked points. An alternative way to prove Theorem~\ref{thm:main_thm} in the (pre-)periodic case is to use the continuity argument which we do in the subsequent article. 
\end{remark}

Next, we prove that $\mathcal{C}_f(\rho)$ is compact.

\begin{thm}[Compactness]
	\label{thm:compact_subset}
	$\mathcal{C}_f(\rho)$ is a compact subset of $\mathcal{T}_f$.
\end{thm}
\begin{proof}
	The proof of the theorem is identical to the proof of \cite[Theorem~5.11]{IDTT2}, except of a small detail that in our case we are allowed to have non-trivial clusters. This obviously does not violate the convergence argument \cite[Theorem~5.11]{IDTT2} since there is almost no relative movement of point inside of clusters outside of $D_\rho$. 
\end{proof}

Now, we have all ingredients for the proof Classification Theorem~\ref{thm:main_thm}.

\begin{proof}[Proof of Classification Theorem~\ref{thm:main_thm}]	
	The proof literally repeats the proof of \cite[Theorem~1.4]{IDTT2} after obvious replacement of invariance and compactness theorems by our upgraded versions: Theorem~\ref{thm:invariant_subset} and Theorem~\ref{thm:compact_subset}.
\end{proof}

\section{Appendix A}
\label{sec:app_A}

In this appendix we provide some properties of polynomials and of their compositions with the exponential. 

The following lemma gives estimates on the coefficients of polynomials $p$ if we know an estimate on the magnitude of singular values of $p\circ\exp$.

\begin{lmm}[Singular values bound coefficients \cite{IDTT2}]
	\label{lmm:bounds_of_coefficients_through_SV} Fix some integer $d>0$.
	There exists a universal constant $L>0$ such that if $g=p\circ\exp$ where $p(z)=z^d+b_{d-1}z^{d-1}+...+b_1 z+b_0$, and singular values of $g$ are contained in the disk $\mathbb{D}_\rho(0)$ with $\rho$ big enough, then $\abs{b_k}<L\rho^\frac{d-k}{d}$.
\end{lmm}

Next statement describes the behavior of polynomials $p$ outside of a disk containing singular values of $p\circ\exp$ subject to the condition that this disk is big enough. 

\begin{lmm}[Preimages of outer disks \cite{IDTT2}]
	\label{lmm:complement_of_big_disk_under_preimage}
	Fix some integer $d>1$. If $\rho>1$ is big enough, then for every $g=p\circ\exp$ with monic polynomial $p$ of degree $d$ and singular values contained in $\mathbb{D}_\rho(0)$ for all $r\geq\rho$ we have
	$$p^{-1}(\mathbb{D}_r(0))\subset\mathbb{D}_r(0).$$
\end{lmm}

For a polynomial $p$ we denote by $\Crit(p)$ the set of its critical points. By $\dist(A,B)$ we denote the Euclidean distance between sets $A,B\subset\mathbb{C}$.

\begin{lmm}[Distance between preimages under polynomials]
	\label{lmm:dist_between_preimages_under_polynomial}
	Let $d>1$ be an integer. Then there exists a real constant $K>0$ such that for every real  $\epsilon\in(0,1)$ and $r>1$ holds the following statement.
	
	If $p$ is a monic polynomial of degree $d$ and $\alpha\in\mathbb{C}$ is such that all roots $z_1,z_2,...,z_d$ of the equation $p(z)=\alpha$ are contained in $\mathbb{D}_r(0)$ and $\dist\left(\{z_1,...,z_d\}, \Crit(p)\right)>\epsilon$, then for every pair of indices $i,j$ such that $1\leq i<j\leq d$ we have 
	$$\abs{z_i-z_j}>K\left(\frac{\epsilon}{r}\right)^{d^2}.$$
\end{lmm}
\begin{proof}
	It is enough to proof the lemma for $i=1, j=2$.
	
	Let $q(z):=p(z)-\alpha=(z-z_1)\cdot...\cdot(z-z_d)$, and let $\Delta_q$ be the discriminant of $q$. Then
	$$\Delta_q=\prod_{i<j}(z_i-z_j)^2=(-1)^{\frac{d(d-1)}{2}} p'(z_1)p'(z_2)\cdot...\cdot p'(z_d).$$
	
	If $p'(z)=d(z-c_1)\cdot...\cdot(z-c_{d-1})$, then for $n\in\{1,...,d\}$ we obtain $\abs{p'(z_n)}\geq d\epsilon^{d-1}$. Hence
	$$\abs{\Delta_q}\geq d^d \epsilon^{d(d-1)}.$$
	
	Finally,
	$$\abs{z_1-z_2}=\frac{\abs{\Delta_q}}{\abs{\Delta_q/(z_1-z_2)}}>\frac{\abs{\Delta_q}}{(2r)^{d(d-1)-1}}\geq\frac{d^d \epsilon^{d(d-1)}}{(2r)^{d(d-1)-1}}>K\left(\frac{\epsilon}{r}\right)^{d^2}$$
	where $K=d^d/2^{d(d-1)-1}.$
\end{proof}

Recall that for an entire function $g\in\mathcal{N}_d$ we denote by $\SV(g)$ the set of its finite singular values.

\begin{lmm}[Distance between preimages under $f\in\nd$]
	\label{lmm:distance_between_preimages_under_function}
	Fix an integer $d>1$. For all $\epsilon\in(0,1)$ and all big enough $\rho>0$ holds the following statement.
	
	If  $p$ is a monic polynomial of degree $d$ and $g=p\circ\exp$ satisfy $\SV(g)\subset\mathbb{D}_\rho(0)$, and $\alpha\in\mathbb{D}_{2\rho}(0)$ is such that $\dist(\alpha,\SV(g))>\epsilon$, then for every pair of distinct complex numbers $w_1,w_2\in g^{-1}(\alpha)$ holds
	$$\abs{w_1-w_2}>\left(\frac{\epsilon}{\rho}\right)^{d^4}.$$  
\end{lmm}
\begin{proof}
	Pick $\zeta\in p^{-1}(\alpha)$, and let $c$ be a critical point of $p$. Due to Lemma~\ref{lmm:complement_of_big_disk_under_preimage}, for all $\rho$ big enough we have $\abs{\zeta}<2\rho$ and $\abs{c}<2\rho$. But then
	$$\abs{\zeta-c}>\frac{\dist(\alpha,\SV(g))}{\max\limits_{\abs{z}<2\rho}\abs{p'(z)}}>\frac{\epsilon}{\max\limits_{\abs{z}<2\rho}\abs{p'(z)}}.$$
	From Lemma~\ref{lmm:bounds_of_coefficients_through_SV} follows that for all $\rho$ big enough
	$$\max\limits_{\abs{z}<2\rho}\abs{p'(z)}<\rho^d.$$
	Hence, if $\zeta_1,\zeta_2\in p^{-1}(\alpha)$, from Lemma~\ref{lmm:dist_between_preimages_under_polynomial} follows that
	$$\abs{\zeta_1-\zeta_2}>K\left(\frac{\dist(p^{-1}(\alpha),\Crit(p))}{2\rho}\right)^{d^2}>\frac{K}{2^{d^2}}\left(\frac{\epsilon}{\rho^{d+1}}\right)^{d^2}.$$
	Further, if $e^{w_1}=\zeta_1$ and $e^{w_2}=\zeta_2$, then for big $\rho$
	$$\abs{w_1-w_2}>\frac{\abs{\zeta_1-\zeta_2}}{\max\limits_{\Re z<\log 2\rho}\abs{\exp'z}}>\frac{K}{2^{d^2}}\left(\frac{\epsilon}{\rho^{d+1}}\right)^{d^2}\frac{1}{2\rho}>\left(\frac{\epsilon}{\rho}\right)^{d^4}.$$
\end{proof}

\section{Acknowledgements}

We would like to express our gratitude to our research team in Aix-Marseille Universit\'e, especially to Dierk Schleicher who supported this project from the very beginning, Sergey Shemyakov who carefully proofread all drafts, as well as to  Kostiantyn Drach, Mikhail Hlushchanka, Bernhard Reinke and Roman Chernov for uncountably many enjoyable and enlightening discussions of this project at different stages. We also want to thank Dzmitry Dudko for his multiple suggestions that helped to advance the project, Lasse Rempe for his long list of comments and relevant questions, and Adam Epstein for important discussions especially in the early stages of this project. 

Finally, we are grateful to funding by the Deutsche Forschungsgemeinschaft DFG, and the ERC with the Advanced Grant “Hologram” (695621), whose support provided excellent conditions for the development of this research project. This research received support from the European Research Council (ERC) under the European Union Horizon 2020 research and innovation programme, grant 647133 (ICHAOS).

\vspace{1em}
Aix-Marseille Universit\'e, France

\textit{Email:} bconstantine20@gmail.com


\begin{thebibliography}{RSSS}
	
%	\bibitem[A]{Ahlfors} Lars Ahlfors, \emph{Lectures on Quasiconformal Mappings}. American Mathematical Society (2006)
	
	\bibitem[B0]{MyThesis} Konstantin Bogdanov, \emph{Infinite-dimensional Thurston theory and transcendental dynamics with escaping singular orbits}. PhD Thesis, Aix-Marseille Universit\'e, 2020.
	
	\bibitem[B1]{IDTT1} Konstantin Bogdanov, \emph{Infinite-dimensional Thurston theory and transcendental dynamics I: infinite-legged spiders}. arXiv:2102.00300
	
	\bibitem[B2]{IDTT2} Konstantin Bogdanov, \emph{Infinite-dimensional Thurston theory and transcendental dynamics II: classification of entire functions with escaping singular orbits}. arXiv:2102.10728
	
%	\bibitem[Ba]{Baranski} Krzysztof Bara\'nski, \emph{Trees and hairs for some hyperbolic entire maps of finite order}. Math Z \textbf{257} (2007), 33--59.
	
%	\bibitem[BA]{BA} Arne Beurling and Lars Ahlfors, \emph{The boundary correspondence under quasiconformal mappings}. Acta Mathematica \textbf{96} (1956), 125–142.
	
%	\bibitem[BH1]{BrannerHubbard1}
%	Bodil Branner and John Hubbard, \emph{The iteration of cubic polynomials, part I: The global topology of parameter space}. Acta Mathematica \textbf{160}.1 (1988), 143--206. 
%	
%	\bibitem[BH2]{BrannerHubbard2}
%	Bodil Branner and John Hubbard, \emph{The iteration of cubic polynomials, part II: Patterns and parapatterns}. Acta Mathematica \textbf{169}.1 (1992), 229--325.
%	
%	\bibitem[Br]{BrownThesis}
%	David Brown, 
%	\emph{Thurston equivalence without postcritical finiteness for a family of polynomial and exponential mappings}. PhD Thesis, Cornell University.
%	
%	
%	
%	\bibitem[CT]{Cui}
%	Guizhen Cui and Lei Tan, \emph{A Characterization of Hyperbolic Rational Maps}. Inventiones Mathematicae \textbf{183}.3 (2011), 451--516.
	
	\bibitem[DH]{DH}
	Adrien Douady and John Hubbard, \emph{A proof of Thurstons's topological characterization of rational functions}. Acta Mathematica \textbf{171} (1993), 263--297.
	
%	\bibitem[DK]{DevaneyKrych}
%	Robert Devaney and Michal Krych, 
%	\emph{Dynamics of $\exp(z)$}. Ergod. Theor. Dyn. Syst. \textbf{4} (1) (1984), 35--52.
	
%	\bibitem[DG]{DouadyGoldberg}
%	Adrien Douady and Lisa Goldberg, \emph{The nonconjugacy of certain exponential functions}, in \emph{Holomorphic functions and moduli I} (Berkeley, CA 1986), Math. Sci. Res. Publ. \textbf{10} (1988), 1--7; Springer, New York, 1988.
%	
%	
%	\bibitem[dMP1]{LauraKevin}
%	Laura deMarco and Kevin Pilgrim,
%	\emph{The classification of polynomial basins of infinity}. Manuscript (2011), arXiv:1107.1091.
%	
%	\bibitem[dMP2]{LauraKevin2}
%	Laura deMarco and Kevin Pilgrim,
%	\emph{Polynomial basins of infinity}. Manuscript (2014), see http://homepages.math.uic.edu/\~{}demarco/basins.pdf
	
%	\bibitem[dS]{hawaiian}
%	Bart de Smit,
%	\emph{The fundamental group of the Hawaiian earring is not free}. International Journal of Algebra and Computation \textbf{2} (1) (1992), 33--37.
%	
%	\bibitem[Ep]{AdamThesis}
%	Adam Epstein, \emph{Towers of Finite Type Complex Analytic Maps}. PhD Thesis, CUNY, New York 1993. Available at {http://www.warwick.ac.uk/\~{}mases/Thesis.pdf}
%	
%	\bibitem[EKT]{EpsteinKeenTresser}
%	Adam Epstein, Linda Keen, and Charles Tresser,
%	\emph{The set of maps $F_{a,b}\colon x\mapsto x+a+(b/2\pi)\sin(2\pi x)$ with any given rotation interval is contractible}.
%	Comm. Math. Phys. \textbf{173} 2 (1995), 313--333.
%	
%	\bibitem[EL]{ErLyu}
%	Alexandre Eremenko and Mikhail Lyubich, \emph{Dynamical properties of some classes of
%		entire functions}. Annales de l’institut Fourier, \textbf{42} (1992) no. 4, 989--1020
	
	\bibitem[F]{MarkusThesis}
	Markus F\"orster, 
	\emph{Exponential maps with escaping singular orbits}. 
	PhD Thesis, International University Bremen, 2006.
%	
%	\bibitem[FM]{Primer}
%	Benson Farb and Dan Margalit,
%	\emph{A primer on mapping class groups}. Princeton University Press (2012).
	
	\bibitem[FRS]{FRS} 
	Markus F\"orster, Lasse Rempe, and Dierk Schleicher,
	\emph{Classification of escaping exponential maps}. Proc Amer Mathe Soc. \textbf{136} (2008), 651--663.
	
	\bibitem[FS]{MarkusParaRays} 
	Markus F\"orster and Dierk Schleicher,  \emph{Parameter rays in the space of exponential maps}. Ergod Theory Dynam Systems \textbf{29} (2009), 515--544.
	
	\bibitem[GL]{Gardiner}
	Frederick P. Gardiner and Nikola Lakic, \emph{Quasiconformal \tei\ theory}. American Mathematical Society (2000).
	
%	\bibitem[H1]{HubbardBook1}
%	John Hubbard, \emph{\tei\ theory and applications to geometry, topology, and dynamics}. Volume 1: \tei\ theory. Ithaca, NY: Matrix Editions (2006).
	
	\bibitem[H]{HubbardBook2}
	John Hubbard, \emph{\tei\ theory and applications to geometry, topology, and dynamics}, Volume 2: Surface homeomorphisms and rational functions. 
	Ithaca, NY: Matrix Editions (2016).
	
%	\bibitem[HS]{Spiders}
%	John Hubbard and Dierk Schleicher, 
%	\emph{The spider algorithm}. 
%	In: R. Devaney (ed.), \emph{Complex dynamics: the mathematics behind the Mandelbrot and Julia sets}. Proc Symp Pure Math \textbf{49}, 
%	Amer Math Soc (1994), 155--180. 
%	
%	
%	\bibitem[HSS]{HSS}
%	John Hubbard, Dierk Schleicher, and Mitsuhiro Shishikura, 
%	\emph{Exponential Thurston maps and limits of quadratic differentials}. Journal Amer Math Soc \textbf{22} (2009), 77--117.
%	
%	\bibitem[Ka]{Karpinska}
%	Bogus{\l}awa Karpi\'nska, \emph{Hausdorff dimension of the hairs without endpoints for $\lambda\exp$}. Compt Rend Acad Sci Paris S\'er I Math \textbf{328} (1999), 1039--1044.
%	
%	\bibitem[LSV]{LSV}
%	Bastian Laubner, Dierk Schleicher, and Vlad Vicol,
%	\emph{A combinatorial classification of postsingularly preperiodic complex exponential maps}. Discr Cont Dyn Syst \textbf{22} 3 (2008), 663--682.
	
%	\bibitem[LV]{LehtoVirtanen}
%	Olli Lehto and K.I. Virtanen, \emph{Quasiconformal mappings in the plane}. Springer-Verlag Berlin Heidelberg (1973).
%	
%	\bibitem[McM]{CurtStrebel}
%	Curt McMullen, 
%	\emph{Automorphisms of rational maps}. 
%	In: Holomorphic functions and moduli, Vol. I (Berkeley, CA, 1986), 31--60.
%	Math. Sci. Res. Inst. Publ. \textbf{10}, Springer Verlag, New York, 1988. 
%	
%	\bibitem[Re1]{LasseParaSpace}
%	Lasse Rempe, \emph{Rigidity of escaping dynamics for transcendental entire functions}. Acta Mathematica \textbf{203} (2009), 235--267.
%	
%	\bibitem[Re2]{LasseSlowEscape}
%	Lasse Rempe, \emph{Topological dynamics of exponential maps on their escaping sets}, Ergod
%	Theory Dynam Systems \textbf{26} (2006), 1939--1975.
	
	\bibitem[RRRS]{RRRS}
	G\"unter Rottenfu{\ss}er, Johannes R\"uckert, Lasse Rempe, and Dierk Schlei\-cher, \emph{Dynamic rays of bounded-type entire functions}. {Annals of Mathematics} \textbf{173} (2011), 77--125.
	
%	\bibitem[RS]{RS-Inventiones}
%	Lasse Rempe and Dierk Schleicher, 
%	\emph{Bifurcations in the space of exponential maps}. 
%	Inventiones Math. \textbf{175} 1 (2009), 103--135.
%	
%	\bibitem[S1]{DS-Duke} 
%	Dierk Schleicher, \emph{The dynamical fine structure of iterated cosine maps and a dimension paradox}. 
%	Duke Math J \textbf{136} 2 (2007), 343--356.
%	
%	\bibitem[S3]{DS-CRAS}
%	Dierk Schleicher, 
%	\emph{Hyperbolic components in exponential parameter space}. 
%	Compt Rend Acad Sci Paris S\'er I Math \textbf{339} (2004), 223--228.
	
%	\bibitem[S2]{DS-Monthly}
%	Dierk Schleicher, 
%	\emph{Hausdorff dimension, its properties, and its surprises}. 
%	Amer. Math. Monthly \textbf{114} 6 (2007), 509--528.
%	
%	\bibitem[S]{DS-Transcendental}
%	Dierk Schleicher,
%	\emph{Dynamics of entire functions}. In: G. Gentili et al (eds), \emph{Holomorphic dynamical systems}. LNS \textbf{1998} (2010), 295--339.
%	
%	\bibitem[Sel]{NikitaThesis}
%	Nikita Selinger, \emph{On Thurston's characterization theorem for branched covers}. PhD Thesis, Jacobs University, 2011.
	
	\bibitem[SZ]{SZ-Escaping}
	Johannes Zimmer and Dierk Schleicher, 
	\emph{Escaping points for exponential maps}. 
	J. Lond. Math. Soc. (2) \textbf{67} (2003), 380--400.
	
\end{thebibliography}
\end{document}